\documentclass[11pt]{article}
\usepackage[centertags,intlimits]{amsmath}

\usepackage{amsfonts}                     
\usepackage{amsthm}
\usepackage{graphics}

\allowdisplaybreaks[2]
\newtheorem{theorem}{Theorem}
\newtheorem{lemma}{Lemma}

\theoremstyle{remark}

\newcommand{\nc}{\newcommand}
\nc{\vb}{\mathbf{v}}
\nc{\bx}{\mathbf{x}}
\nc{\by}{\mathbf{y}}
\nc{\bz}{\mathbf{z}}
\nc{\bu}{\mathbf{u}}
\nc{\bv}{\mathbf{v}}
\nc{\ba}{\mathbf{a}}
\nc{\bs}{\mathbf{s}}
\nc{\bq}{\mathbf{q}}
\nc{\bd}{\mathbf{d}}
\nc{\bb}{\mathbf{b}}
\nc{\bc}{\mathbf{c}}
\nc{\bi}{\mathbf{i}}
\nc{\bfr}{\mathbf{r}}
\nc{\bA}{\mathbf{A}}
\nc{\R}{\mathbb R}
\nc{\N}{\mathbb N}
\nc{\C}{\mathbb C}
\nc{\D}{\mathbb D}
\nc{\Z}{\mathbb Z}
\nc{\F}{\mathbf F}
\nc{\bbS}{\mathbb S}
\nc{\bE}{\mathbf E}
\nc{\B}{\cal B}
\nc{\br}{\bigr}
\nc{\bl}{\bigl}
\nc{\Bl}{\Bigl}
\nc{\Br}{\Bigr}
\nc{\ind}[1]{\,\mathbf{1}_{\{#1\}}\,}
\nc{\bP}{\mathbf{P}}

\title{Some properties of substochastic matrices} 
\author{A. Puhalskii}

\begin{document}
\maketitle

\begin{abstract}
In this note we establish some properties of  matrices
that we haven't been able to find in the literature.
\end{abstract}
Let  $I$ represent the  $n\times n$--identity matrix and let 
 $P=(p_{lm})$ represent an $n\times n$ row substochastic  matrix of spectral radius less
than unity. The next two assertions seem to concern new properties
of substochastic matrices,  cf., Bellman \cite{Bel70},
 Gantmacher \cite{Gan98} and Lancaster
\cite{Lan69}.
\begin{theorem}
  The diagonal elements of $(I-P^T)^{-1}$ are maximal elements of
  their respective rows.
\end{theorem}
\begin{proof}
Let $P$ be of size $n\times n$\,.
Let $C=(c_{ml})$ be defined by $C=(I-P^T)^{-1}$\,. We have that
\begin{equation*}
  c_{ml}=\frac{1}{\text{det}\,(I-P^T)}\,(-1)^{m+l}M_{lm}\,,
\end{equation*}
where $M_{lm}$ represents the $(l,m)$--minor of  $I-P^T$\,. We note that
$\text{det}\,(I-P^T)>0$\,. Indeed, $\text{det}\,I=1$ and
$\text{det}\,(I-\lambda P^T)\not=0$ for $\lambda\in[0,1]$ by $P$ being
of spectral radius less than one.
By continuity, $\text{det}\,(I-P^T)>0$\,. 
Thus, one needs to prove that
$(-1)^{m+l}M_{lm}\le M_{mm}$\,.
Suppose that $l=m+1$\,. Then one needs that
$M_{mm}+M_{m+1,m}\ge0$\,. By the determinant being multilinear,
 $M_{mm}+M_{m+1,m}$ is the determinant of
matrix  $\tilde I-\tilde P$\,, where $\tilde I$ 
is the identity $(n-1)\times (n-1)$--matrix and $\tilde P$ is the  $(n-1)\times
(n-1)$--matrix that is obtained from  $P^T$ by adding up rows $m$ and
$m+1$ and deleting the $m$th column.
As  $P^T$ is column substochastic, 
$\tilde P$ is column substochastic as well. Hence, it's of spectral
radius less than or equal to one. It follows that
 $\text{det}\,(\tilde I-\tilde P)\ge0$\,, so, 
$M_{m+1,m}+ M_{m,m}\ge0$\,. Suppose that $l>m+1$\,. By transposing
adjacent rows and columns one can move row $l$ of $\tilde I-\tilde P$
 into the position of row $m+1$ and move
column $l$ into the position of 
column $m+1$\,, respectively, without disturbing the
order in which the other rows and columns are arranged. The matrix thus obtained
is of the form $\tilde I-\tilde{\tilde P}$\,.
For this matrix, $\tilde M_{m+1,m}=(-1)^{l-m-1}M_{lm}$\,, as the minor
sign will flip only when the columns are transposed. Also
$\tilde M_{mm}=(-1)^{2(l-m-1)}M_{mm}$\,. Since $\tilde M_{mm}\ge
(-1)\tilde M_{m+1,m}$\,, we conclude that
 $M_{mm}\ge(-1)^{l-m-2} M_{lm}=(-1)^{l+m}
M_{lm}$\,.
The case where  $l<m$ is dealt with similarly.
\end{proof}
For square matrix $B$\,,
let   $B(i|j)$ denote the matrix that is obtained from
 $B$ by deleting the $i$th row and the  $j$th column.
Let $b_{\cdot
      l}$
represent the
 $l$th column of $B$ with the  $l$th entry deleted and let
    $b_{l\cdot}$ represent the  $l$th row of $B$ with $l$th entry deleted.
Let 
$e_i$ represent the $i$th element of the standard basis in
$\R^{n-1}$ and let
for $l\not=m$
\begin{equation*}
  f_{ml}=
  \begin{cases}
    e_m^T\,,&\text{ если }m<l\,,\\
e_{m-1}^T\,,&\text{ если }m>l\,,
  \end{cases}
\end{equation*}
\begin{theorem}
  The following identites hold:
\begin{equation*}
  \frac{p_{m\cdot}\bl((I-P)(m|m)\br)^{-1}p_{\cdot m}}{1-p_{mm}
-p_{m\cdot}\bl((I-P)(m|m)\br)^{-1}p_{\cdot m}}
=\sum_{k\not=m}\frac{p_{km}f_{mk}\bl((I-P)(l|l)\br)^{-1}p_{\cdot k}}{1-p_{kk}
-p_{k\cdot}\bl((I-P)(l|l)\br)^{-1}p_{\cdot k}}
\end{equation*}
and
\begin{multline*}
    \frac{(1-p_{mm})f_{lm}((I-P)(m|m))^{-1}p_{\cdot m}}{1-p_{mm}
-p_{m\cdot}((I-P)(m|m))^{-1}p_{\cdot m}}
=\frac{p_{lm}}{1-p_{ll}-p_{l\cdot}((I-P)(l|l))^{-1}p_{\cdot l}}\\
+\sum_{\substack{k\not=l,\\k\not=m}}\frac{p_{km}f_{lk}((I-P)(k|k))^{-1}p_{\cdot
    k}}
{1-p_{kk}-p_{k\cdot}((I-P)(k|k))^{-1}p_{\cdot k}}
\end{multline*}
\end{theorem}
The assertion of the theorem is a special case of the following  result.
\begin{theorem}
\label{le:tozd}
Let 
$B=(b_{ij})$ be an  $n\times n$--matrix with nonzero principal minors.
  The following identities hold:
\begin{equation}
  \label{eq:13}
    \frac{b_{m\cdot}\bl(B(m|m)\br)^{-1}b_{\cdot m}}{b_{mm}
-b_{m\cdot}\bl(B(m|m)\br)^{-1}b_{\cdot m}}
=\sum_{l\not=m}\frac{b_{lm}f_{ml}\bl(B(l|l)\br)^{-1}b_{\cdot l}}{b_{ll}
-b_{l\cdot}\bl(B(l|l)\br)^{-1}b_{\cdot l}}
\end{equation}
and
\begin{equation}
  \label{eq:20}
  -\,  \frac{b_{mm}f_{lm}(B(m|m))^{-1}b_{\cdot m}}{b_{mm}
-b_{m\cdot}(B(m|m))^{-1}b_{\cdot m}}
=-\,\frac{b_{lm}}{b_{ll}-b_{l\cdot}(B(l|l))^{-1}b_{\cdot l}}\\
+\sum_{\substack{k\not=l,\\k\not=m}}\frac{b_{km}f_{lk}(B(k|k))^{-1}b_{\cdot
    k}}
{b_{kk}-b_{k\cdot}(B(k|k))^{-1}b_{\cdot k}}\,.
\end{equation}
\end{theorem}
We precede the proof with two lemmas.
Let $\text{adj}$ be used to denote the adjoint matrix and let
$M_{ij}(l|l)$ denote the  $(i,j)$--minor of the 
matrix   $B(l|l)$\,.
\begin{lemma}
  \label{le:adj}
If  $l\not=m$\,, then
\begin{equation*}
  f_{ml}\,\text{adj}(B(l|l))b_{\cdot l}=(-1)^{m+l+1}\,\text{det}\,(B(l|m))\,.
\end{equation*}
\end{lemma}
\begin{proof}
  Suppose that $l>m$\,. We have that
\begin{equation*}
  e_m^T\,\text{adj}\,(B(l|l))b_{\cdot l}
=\sum_{j=1}^{l-1}(-1)^{m+j}M_{jm}(l|l)b_{jl}+
\sum_{j=l+1}^{n}(-1)^{m+j-1}M_{j-1,m}(l|l)b_{jl}\,.
\end{equation*}
Since $M_{jm}(l|l)=M_{j,l-1}(l|m)$ when $j\le l-1$ and
$M_{j-1,m}(l|l)=M_{j-1,l-1}(l|m)$ when $j\ge l+1$\,, we have that
$e_m^T\,\text{adj}\,(B(l|l))b_{\cdot l}=(-1)^{m+l+1}\text{det}\,(B(l|m))\,.$

Suppose that $l<m$\,. Similarly to the above,
\begin{multline*}
e_{m-1}^T\,\text{adj}\,(B(l|l))b_{\cdot l}=
\sum_{j=1}^{l-1}(-1)^{m+j-1}M_{j,m-1}(l|l)b_{jl}+
\sum_{j=l+1}^n (-1)^{m+j}M_{j-1,m-1}(l|l)b_{jl}\\
=\sum_{j=1}^{l-1}(-1)^{m+j-1}M_{jl}(l|m)b_{jl}+
\sum_{j=l+1}^n (-1)^{m+j}M_{j-1,l}(l|m)b_{jl}=(-1)^{m+l+1}\,\text{det}\,(B(l|m))\,.
\end{multline*}
\end{proof}
\begin{lemma}
  \label{le:det}
For arbitrary $l=1,2,\ldots,n$\,,  \begin{equation*}
    b_{ll}\,\text{det}(B(l|l))-b_{l\cdot}\text{adj}(B(l|l))b_{\cdot
      l}=\text{det}\,(B)\,.  
  \end{equation*}
\end{lemma}
\begin{proof}
  Since
  \begin{equation*}
    b_{l\cdot}\,\text{adj}(B(l|l))b_{\cdot l}=
\sum_{j\not=l}b_{lj}f_{jl}\,\text{adj}\,(B(l|l))b_{\cdot l}\,,
  \end{equation*}
an application of Lemma \ref{le:adj} yields
\begin{equation*}
      b_{l\cdot}\,\text{adj}(B(l|l))b_{\cdot l}
=\sum_{j\not=l}b_{lj}(-1)^{j+l+1}\,\text{det}\,(B(l|j))=
-\text{det}\,(B)+b_{ll}\,\text{det}\,(B(l|l))\,.
\end{equation*} \end{proof}
\begin{proof}[Proof of Theorem \ref{le:tozd}]
  Equation \eqref{eq:13} holds if and only if
\begin{equation}
  \label{eq:17}
      \frac{b_{m\cdot}\,\text{adj}\,(B(m|m))b_{\cdot
          m}}{b_{mm}\,\text{det}\,(B(m|m)) 
-b_{m\cdot}\,\text{adj}\,(B(m|m))b_{\cdot m}}
=\sum_{l\not=m}\frac{b_{lm}f_{ml}\,\text{adj}\,(B(l|l))
b_{\cdot l}}{b_{ll}\,\text{det}\,(B(l|l))
-b_{l\cdot}\,\text{adj}\,(B(l|l))b_{\cdot l}}\,.
\end{equation}
By Lemma \ref{le:det}, the denominators in 
 \eqref{eq:17} equal $\text{det}(B)$\,.
One thus needs to prove that 
\begin{equation*}
  b_{m\cdot}\,\text{adj}\,(B(m|m))b_{\cdot
          m}
=\sum_{l\not=m}b_{lm}f_{ml}\,\text{adj}\,(B(l|l))
b_{\cdot l}\,.
\end{equation*}
By Lemma \ref{le:adj}, 
\begin{equation*}
  \sum_{l\not=m}b_{lm}f_{ml}\,\text{adj}\,(B(l|l))
b_{\cdot l}=\sum_{l\not=m}b_{lm}(-1)^{m+l+1}\,\text{det}(B(l|m))=
-\text{det}(B)+b_{mm}\,\text{det}(B(m|m))\,,
\end{equation*}
which concludes the proof of \eqref{eq:13} by Lemma  \ref{le:det}.

Multiplying the numerators and denominators in 
 \eqref{eq:20} with the determinants of the matrices being inverted
and applying Lemma \ref{le:det} obtain that 
 \eqref{eq:20} is equivalent to the equation
\begin{equation}
  \label{eq:21}
  -b_{mm}f_{lm}\,\text{adj}\,(B(m|m))b_{\cdot m}
=-b_{lm}\,\text{det}\,(B(l|l))
+\sum_{\substack{k\not=l,\\k\not=m}}b_{km}f_{lk}\,\text{adj}\,(B(k|k))b_{\cdot
    k}\,.
\end{equation}
By Lemma \ref{le:adj},
\begin{equation*}
  \sum_{k\not=l}b_{km}f_{lk}\,\text{adj}\,(B(k|k))b_{\cdot
    k}=\sum_{k\not=l}b_{km}
(-1)^{k+l+1}\,\text{det}\,(B(k|l))\,.
\end{equation*}
Since $\sum_{k=1}^nb_{km}
(-1)^{k+l}\,\text{det}\,(B(k|l))=0$ because the lefthand side 
is the determinant of the matrix that is obtained from the matrix $B$ 
by replacing the 
$l$th column with the   $m$th column, 
\begin{equation*}
  \sum_{k\not=l}b_{km}f_{lk}\textit{adj}(B(k|k))b_{\cdot
    k}=b_{lm}\textit{det}(B(l|l))\,.
\end{equation*}
Therefore,
\begin{equation*}
  \sum_{\substack{k\not=l,\\k\not=m}}b_{km}f_{lk}\,\text{adj}\,(B(k|k))b_{\cdot
    k}=b_{lm}\,\text{det}\,(B(l|l))-b_{mm}f_{lm}\,\text{adj}\,(B(m|m))b_{\cdot
    m}\,.
\end{equation*}
Equation \eqref{eq:21} is proved.

\end{proof}

\end{document}